\documentclass[a4paper,12pt]{amsart}

\usepackage{fullpage}

\usepackage{amssymb,amsmath,amsthm}
\usepackage{bm}

\usepackage[english]{babel}
\usepackage[utf8]{inputenc}

\usepackage[table]{xcolor}
\usepackage{graphicx,enumerate,url}

\usepackage{tikz}
\usetikzlibrary{calc,through,backgrounds,shapes,matrix}

\usepackage{array}
\newcolumntype{C}[1]{>{\centering\arraybackslash$}p{#1}<{$}}
\usepackage{multirow}
\usepackage{hhline}

\usepackage{hyperref}

\usepackage{url, hypcap}
\hypersetup{colorlinks=true, citecolor=darkblue, linkcolor=darkblue}

\usepackage[colorinlistoftodos]{todonotes}
\numberwithin{equation}{section}

\usepackage{cleveref}
\crefname{conjecture}{conjecture}{conjectures}
\Crefname{conjecture}{Conjecture}{Conjectures}
\crefname{observation}{observation}{observations}
\Crefname{observation}{Observation}{Observations}
\crefname{hope}{hope}{hopes}
\Crefname{hope}{Hope}{Hopes}
\crefname{openproblem}{open problem}{open problems}
\Crefname{openproblem}{Open problem}{Open problems}

\newcommand{\drawPath}[4]
{%
  \draw[rounded corners=1, line width=#2, #3] #4
  \foreach \dir in {#1}%
  {
    \ifnum\dir=0
    -- ++(1,0)
    \else
      \ifnum\dir=1
        -- ++(0,1)
      \else
        \ifnum\dir=2
          -- ++ (1,1)
        \fi
      \fi
    \fi
  };}

\newtheorem{theorem}{Theorem}[section]
\newtheorem{proposition}[theorem]{Proposition}
\newtheorem{lemma}[theorem]{Lemma}
\newtheorem{corollary}[theorem]{Corollary}
\newtheorem{definition}[theorem]{Definition}

\theoremstyle{definition}
\newtheorem{example}[theorem]{Example}
\newtheorem{remark}[theorem]{Remark}
\newtheorem{openproblem}{Open problem}

\definecolor{darkblue}{rgb}{0,0,0.7} %
\newcommand{\Def}[1]{\textbf{#1}} %
\newcommand{\Dfn}[1]{\textbf{#1}} %

\renewcommand{\emptyset}{\varnothing}
\newcommand{\N}{\mathbb{N}}
\newcommand{\Z}{\mathbb{Z}}

\newcommand{\R}{\mathbb{R}}

\newcommand{\Se}{S^{\e}}
\newcommand{\Sn}{S^{\n}}

\renewcommand{\mid}{{ \;:\; }}
\newcommand{\bigset}[2]{\left\{\; #1 \;:\; #2 \;\right\}} %

\newcommand{\paths}[1]{\mathcal{L}[#1]}
\newcommand{\latpaths}[1]{\mathcal{L}_{#1}}

\newcommand{\n}{\mathbf{n}}
\newcommand{\east}[1]{\mathbf{E}(#1)}
\newcommand{\e}{\mathbf{e}}
\renewcommand{\d}{\mathbf{d}}
\newcommand{\emp}{\bm\varepsilon}
\newcommand{\dem}[2]{\operatorname{dem}_{#1}(#2)}
\newcommand{\mar}[2]{\operatorname{mark}_{#1}(#2)}
\newcommand{\matroid}[1]{M[#1]}
\newcommand{\bases}[1]{\mathcal{B}[#1]}

\newcommand{\stat}{\operatorname{st}}

\newcommand{\bij}{\Lambda}

\newcommand{\leqcomp}{\leq_{\operatorname{comp}}}

\newcommand{\Ltriv}{{L^{tr}}}

\newcommand{\x}{\mathbf{x}}

\title[Standard complexes of matroids]{Standard complexes of matroids\\ and lattice paths}

\author[A.~Engström]{Alexander Engström} %

\address[A.~Engström]{Department of Mathematics and Systems Analysis, Aalto
University, Espoo, Finland}
\email{alexander.engstrom@aalto.fi}

\author[R.~Sanyal]{Raman Sanyal} %
\address[R.~Sanyal]{Institut f\"ur Mathematik, Goethe-Universit\"at
Frankfurt, Germany} 
\email{sanyal@math.uni-frankfurt.de}

\author[C.~Stump]{Christian Stump$^{\ddagger}$}
\address[C.~Stump]{Fakult\"at f\"ur Mathematik, Ruhr-Universit\"at Bochum, Germany}
\email{christian.stump@rub.de}
\thanks{$^\ddagger$Supported by DFG grant STU 563/4-1 ``Noncrossing phenomena in Algebra and Geometry''.}

\date{\today}


\begin{document}

\begin{abstract}
    Motivated by Gr\"obner basis theory for finite point configurations, we
    define and study the class of \emph{standard complexes} associated to a
    matroid. Standard complexes are certain subcomplexes of the
    independence complex that are invariant under matroid duality.
    For the lexicographic
    term order, the standard complexes satisfy a deletion-contraction-type
    recurrence. We explicitly determine the lexicographic standard complexes
    for lattice path matroids using classical bijective combinatorics.
\end{abstract}

\maketitle

\setcounter{tocdepth}{1}

\newcommand\Ind{\mathcal{I}}%
\section{Introduction}\label{sec:intro}
Matroids come with a rich enumerative theory. Mostly, this can be attributed
to the deletion-contraction paradigm that is inherent to matroid theory and
that culminates in the existence of the Tutte polynomial;
see~\cite[Ch.~6]{matroid_applications}. Other objects encoding valuable
enumerative properties can be associated to matroids. Trivially, the
collection of independent sets $\Ind(M)$ of a matroid $M$ is a simplicial
complex, whose number of faces of various dimensions and topological features
give enumerative invariants of $M$. Far less trivial are Brylawski's broken
circuit complexes~\cite{Bry}. Built on ideas of Whitney, these are simplicial
complexes associated to matroids with a totally ordered groundset. Their
enumerative and topological properties explain many combinatorial
characteristics of the underlying matroid~\cite[Ch.~7]{matroid_applications}.
Both complexes, the independence as well as the broken circuit complex, can be
constructed from a deletion-contraction-type process. The goal of this paper
is to define and study a new class of simplicial complexes associated to
matroids and to showcase their combinatorial structure.

\newcommand\ZO{\{0,1\}}%
\newcommand\Sc{\mathcal{S}}%
\newcommand\ScP{\mathcal{S}_\preceq}%
Our simplicial complexes are motivated by Gr\"obner bases theory: Let $V
\subseteq \ZO^n$ be a $0/1$-point configuration and $\preceq$ a term order on
$\R[x_1,\dots,x_n]$. The standard monomials of the vanishing ideal $I(V)$ are
squarefree and thus encode a simplicial complex that we call  the
\Def{standard complex} $\ScP(V)$ of $V$. A matroid $M$ on groundset $[n] :=
\{1,2,\dots,n\}$ is canonically represented by its basis configuration $V_M
\subset \ZO^n$ and we define its standard complex as $\ScP(M) := \ScP(V_M)$.

Encoding and studying combinatorial objects by means of zero-dimensional
ideals and Gr\"obner bases has a long history; see~\cite{lovasz,
deLoera,robbiano98, Anstee, Hegedus,Laurent07} for a non-exhaustive
selection. In particular~\cite{Laurent07} emphasizes the use of standard
monomials but the point of view of simplicial complexes and matroids has been
largely neglected. Let us highlight some of the interesting
properties of standard complexes of matroids: For any term order $\preceq$,
the standard complex $\ScP(M)$ is a subcomplex of the independence complex
$\Ind(M)$  of $M$ (Corollary~\ref{cor:sub_ind}) and its number of faces is
precisely the number of bases of $M$ (Corollary~\ref{cor:sc_bases}). Moreover,
the standard complex is invariant under matroid duality, that is, $\ScP(M) =
\ScP(M^*)$, where $M^*$ is the matroid dual to $M$
(Proposition~\ref{prop:dual}). 

\newcommand\Scl{\Sc_{\mathsf{lex}}}%
For a nonempty matroid $M$, we write $m(M)$ for the largest element in its groundset $E \subseteq \N$, and we write $\Scl(M)$ for the standard complex with respect to the lexicographic term order such that $x_1 \succ x_2 \succ \cdots$.
Then the standard complex has the following
deletion-contraction-type decomposition. Recall that for a simplicial complex
$K$ and $v$ not a vertex of $K$, the \Def{cone} of $K$ with apex $v$ is the
complex $v \ast K  :=  K \ \cup \ \{ v \cup \sigma : \sigma \in K \}$.

\begin{theorem}\label{thm:lex}
    Let $M$ be a matroid and $m = m(M)$. If $m$ is not a loop or coloop,
    then
    \[
        \Scl(M) \ = \ 
        \Scl(M \backslash m) \ \cup \
        \Scl(M / m) \ \cup \
        \left( m \ast \big(\Scl(M \backslash m) \cap \Scl(M / m) \big) \right) \, .
    \]
    Otherwise, we have
    \[
    \Scl(M) \ =\ \begin{cases}
      \Scl(M \backslash m) &\text{ if $m$ is a coloop} \\
      \Scl(M / m) &\text{ if $m$ is a loop} \\
                \end{cases}\ .
    \]
\end{theorem}
If $M = \{\emptyset\}$, then $\Scl(M) = \{ \emptyset \}$.
This gives a recursive definition of $\Scl(M)$.

\newcommand\map{\mathrm{Map}}%
Recall that for a map $f : X \to Y$ between topological spaces, the
\Def{mapping cone} is the topological space $Y \sqcup_{f(X)} \mathrm{cone}(f(X))$.
Thus Theorem~\ref{thm:lex} states that $\Scl(M)$ is the mapping cone
associated to the inclusion
\begin{equation}\label{eqn:mapping}
        \Scl(M \backslash m) \cap \Scl(M / m)  \ \hookrightarrow \
        \Scl(M \backslash m) \ \cup \
        \Scl(M / m) \, .
\end{equation}

\newcommand\Bases[1]{\mathcal{B}(#1)}%
\newcommand\Mclass{\mathcal{M}}%
\newcommand\Nclass{\Mclass'}%
Theorem~\ref{thm:lex} also has the following combinatorial consequence.  Let
$\Mclass$ be the collection of matroids with groundsets contained in $\N$ and
write $\Bases{M}$ for the collection of bases of a matroid $M \in \Mclass$.

\begin{corollary}\label{cor:bijection}
  There is a unique family $\{ \bij_M \}_{M \in \Mclass}$ of bijections
  \[
    \bij_M : \Bases{M} \to \Scl(M) %
  \]
  such that for every $M \in \Mclass$, $m = m(M)$, and $B \in \Bases{M}$ one has
  \begin{align*}
    \bij_M(B)\ &\subseteq\ B \\
    \bij_{M}(B) \ &= \ \bij_{M \backslash m}(B) &\text{if } m \not \in B, \\
    \bij_{M}(B) \setminus m \ &= \ \bij_{M/m}(B \backslash m ) & \text{if } m \in B .
  \end{align*}
\end{corollary}

In Section~\ref{sec:lex}, we prove \Cref{thm:lex} and recursively
construct the bijection in \Cref{cor:bijection} explicitly.

We call a subclass $\Nclass \subseteq \Mclass$ of matroids such that $M
\backslash m(M), M / m(M) \in \Nclass$ for every $M \in \Nclass$ \Def{max-minor-closed} and observe that \Cref{cor:bijection} provides a tool to understand the standard complex for matroids in the class $\Nclass$ by explicitly constructing a family $\{\bij_M\}_{M\in\Mclass'}$ of bijections and then showing that this family has the proposed properties.

In the second part of the paper, we implement this idea in the case of the
(max-)minor-closed class of lattice path matroids. Lattice paths in $\Z^2$
with only north or east steps are partially ordered by `staying weakly below'.
It was noted by Bonin, de Mier, and Noy~\cite{BMN} that the collection of
lattice paths between two bounding paths
$U$ (for upper boundary) and $L$ (for lower boundary) with common endpoints
are the bases for a (transversal) matroid, called a \Def{lattice path
matroid}. The combinatorics of lattice path matroids (and lattice path in
general) has been of considerable interest in recent years; see, for
example,~\cite{Ard2003, BdM,  Schweig2, Schweig1, ER, Kolja}.  It turns out
that our construction of compatible bijections on the (max-)minor-closed class
of lattice path matroids relates to classical bijective combinatorics.  We
explicitly describe the family $\bij_M$ in this case in \Cref{sec:lattices};
it can be nicely described in terms of lattice path combinatorics and is a
well-known statistic in certain cases (\Cref{thm:mainsimple,,thm:mainhard}).
Surprisingly, we will observe that the bijection in the case of lattice path
matroids does only depend on the lower boundary rather than on the complete
set of bases in this case (\Cref{cor:mainhard}).  For trivial lower and upper
boundaries, the standard complex in question appeared already in the literature
and was discussed in~\cite{Anstee}, see also~\cite[Thm.~1.2]{Hegedus}.

\section{0/1-configurations and standard complexes}
\label{sec:general}

\newcommand\init{\mathrm{in}_{\preceq}}%
\newcommand\Mon{\mathrm{Mon}}%
\newcommand\std{\mathrm{Std}}%
Let $\R[\x] = \R[x_1,\dots,x_n]$ be the polynomial ring in $n$ variables.  A
\Def{term order} on the collection of monomials $\Mon_n = \{ \x^\alpha :
\alpha \in \N^n\}$ is a total order $\preceq$ such that $0 \preceq \x^\alpha$
and $\x^\alpha \preceq \x^\beta$ implies $\x^\alpha \x^\gamma \preceq \x^\beta
\x^\gamma$ for all $\alpha,\beta,\gamma \in \N^n$.  For a given term order
$\preceq$ and polynomial $f \neq 0$, we denote by $\init(f)$ the leading term
of $f$ and for an ideal $I \subseteq \R[\x]$ the \Dfn{initial ideal}
$\init(I)$ is the monomial ideal spanned by $\{ \init(f) : f \in I, f
\neq 0 \}$.  The \Dfn{standard monomials} of $I$ with respect to $\preceq$ are
$\std_\preceq(I) := \{ \x^\alpha : \x^\alpha \not\in \init(I) \}$. This is a
possibly infinite collection of monomials closed under divisibility. If $I =
I(A)$ is the vanishing ideal of a algebraic variety $A \subset \R^n$, then
$\std_\preceq(I)$ is
a basis for the vector space of polynomial functions $\R[A] := \R[\x]/I$ on
$A$.  Every finite collection of points $V \subset \R^n$ is an algebraic
variety and the knowledge of $\std_\preceq(I(V))$ is of the essence for
example in polynomial optimization~\cite{Laurent07} or the design of
experiments~\cite{robbiano98}. We refer the reader to the wonderful
book~\cite{CLO} for Gr\"obner basis theory in general.

If $V \subseteq \ZO^n$ is a \Def{$\mathbf{0/1}$-point configuration}, then $x_i^2 - x_i$
vanishes on $V$ for every $i$. Independent of the term order, we have $x^2_i
\in \init(I(V))$ which shows that 
\begin{proposition}
    If $V$ is a $0/1$-point configuration, then $\std_\preceq(I(V))$ is a collection
    of squarefree monomials.
\end{proposition}

For $\tau \subseteq [n]$, let us write $\x^\tau = \prod_{i \in \tau} x_i$.  In
light of the previous proposition, we define the \Def{standard complex} of a
$0/1$-configuration $V$ with respect to a term order $\preceq$ as
\[
    \ScP(V) \ := \ \{ \tau \subseteq [n] : \x^\tau \in \std_\preceq(I(V)) \} \, .
\]
For a finite configuration $V \subset \R^n$, every function $f : V \to \R$
is the restriction of a polynomial on $\R^n$. This shows that $\R[V] \cong
\R^V$ and gives us
\begin{proposition}\label{prop:card_std}
    For any $0/1$-configuration $V \subseteq \ZO^n$, the standard complex
    $\ScP(V)$ is a simplicial complex on $[n]$ with $|\ScP(V)| = |V|$.
\end{proposition}

Let $M$ be a rank-$r$ matroid on groundset $[n]$ and let $\Bases{M}$ be its
collection of bases. The \Def{basis configuration} of $M$ is the point
configuration
\[
    V_M \ := \ \{ \e_B  \in \ZO^n : B \in \Bases{M} \} \, ,
\]
where $\e_B$ is the characteristic vector of $B \subseteq [n]$. This is
precisely the collection of vertices of the matroid base polytope $P_M$ of $M$
that gives a prominent geometric representation of $M$;
see~\cite{edmonds,postnikov}. For a term order $\preceq$, we write
$\ScP(M) := \ScP(V_M)$. We record the following consequence of
Proposition~\ref{prop:card_std}.

\begin{corollary}\label{cor:sc_bases}
    For any matroid $M$ we have $|\ScP(M)| = |\Bases{M}|$ for any term order
    $\preceq$.
\end{corollary}

\newcommand\circuits{\mathcal{C}}%
Recall that the circuits $\circuits(M)$ of $M$ are
the inclusion-minimal sets $C \subseteq [n]$ such that $C \not\in \Ind(M)$.

\begin{proposition}[{\cite[Prop.~4.3.]{GS17}}]\label{prop:V_M-ideal}
    Let $M$ be rank-$r$ matroid on groundset $[n]$. Then the vanishing ideal
    of $V_M$ is generated by the polynomials $x_i^2 - x_i$ for $i = 1,\dots,n$
    as well as $x_1 + \cdots + x_n - r$ and $\x^C$ for $C \in \circuits(M)$.
\end{proposition}

Proposition~\ref{prop:V_M-ideal} yields the following result. 

\begin{corollary}\label{cor:sub_ind}
    Let $M$ be a matroid on groundset $[n]$ and $\preceq$ a term order. Then
    $\ScP(M) \subseteq \Ind(M)$ is a subcomplex.
\end{corollary}
\begin{proof}
    Note that $\x^C \in \init(I(V_M))$ for $C \in \circuits(M)$ independent of
    the term order. Thus $\ScP(M)$ is a collection of sets $\tau \subseteq
    [n]$ such that $C \not\subseteq \tau$ for all $C \in \circuits(C)$. This
    means that $\tau$ is an independent set of $M$.
\end{proof}

The matroid $M^*$ dual to $M$ is the matroid with bases $\Bases{M^*} = \{ [n]
\setminus B : B \in \Bases{M} \}$.

\begin{proposition}\label{prop:dual}
    Let $M$ be a matroid and $\preceq$ a term order. Then $\ScP(M) =
    \ScP(M^*)$.
\end{proposition}

It follows from the definition that $V_{M^*} = (1,\dots,1) - V_M$. The next
lemma then yields Proposition~\ref{prop:dual}.
Denote by $T_i$ the reflection in the hyperplane $\{\x : x_i = \frac{1}{2} \}$. Thus $T_i$ maps $v \in \ZO^n$  to $(v_1,\dots,v_{i-1}, 1 - v_i,
v_{i+1},\dots, v_n) \in \ZO^n$.

\begin{lemma}\label{lem:SC_reflection}
    Let $V \subset \ZO^n$ be a $0/1$-configuration. Then 
    $\ScP(V) = \ScP(T_i(V))$ for every $i = 1,\dots,n$.
\end{lemma}        
\begin{proof}
    Note that $T_i$ acts on $\R[\x]$ by $T_i(f)(x_1,\dots,x_n) =
    f(x_1,\dots,x_{i-1},1-x_i,x_{i+1},\dots,x_n)$. In particular $I(T_i(V)) =
    T_i(I(V))$. Let $\x^\alpha$ be a monomial. Then $T_i(\x^\alpha) -
    \x^\alpha$ is a polynomial all whose terms strictly divide $\x^\alpha$ and
    thus are strictly smaller in any term order. It follows that $\init(I(V))
    \subseteq \init(I(T_i(V)))$ and thus $\ScP(V) \subseteq \ScP(T_i(V))$.
    However $|V| = |T_i(V)| = |\ScP(T_i(V))| \ge |\ScP(V)| = |V|$, which
    proves the claim.
\end{proof}

\section{Lex order, mapping cones, and bijections}\label{sec:lex}

Let $\preceq$ be the lexicographic term order with $x_1 \succ x_2 \succ \cdots
\succ x_n$. That is $\x^\alpha \prec \x^\beta$ if for the smallest $i$ for
which $\alpha_i \neq \beta_i$ we have $\alpha_i > \beta_i$. The lexicographic
order is strongly tied to elimination in ideals and projections of algebraic
sets~\cite[Ch.~3]{CLO} and lends itself to inductive or recursive arguments.

For a $0/1$-configuration $V \subseteq \ZO^n$, we define
\[
    V^a \ := \ \{ v \in \ZO^{n-1} : (v,a) \in V \} \quad \text{ for } a = 0,1
    .
\]

Following lemma gives a recursive description of the standard monomials for a
$0/1$-point configuration. The recursive structure of standard monomials with
respect to the lexicographic term order has been noted in various contexts;
see, for example~\cite{Lederer} for the description of the standard monomials
for general point sets and the discussions in the last section of that paper.
In the context of VC-dimensions, Anstee, R\'{o}nyai, and Sali call $\Scl(V)$
\emph{order shattering} and give a similar recursive description (and proof)
in Theorem 4.3 of~\cite{Anstee}.

\begin{lemma}\label{lem:lex}
    Let $V \subset \ZO^n$ be a non-empty $0/1$-configuration and let $\prec$
    be the lexicographic term order with $x_1 \succ x_2 \succ \cdots \succ
    x_n$.  Then
    \[
        \Scl(V) \ = \ 
        \Scl(V^0) \ \cup \
        \Scl(V^1) \ \cup \
        ( n \ast (\Scl(V^0) \cap \Scl(V^1) ) ) \, .
    \]
\end{lemma}
\begin{proof}
    \renewcommand\S{\mathcal{S}'}%
    Let us denote the right-hand side by $\S$. We first show that if $\tau
    \not\in \S$, then there is a polynomial $f_\tau \in I(V)$ with
    $\init(f_\tau) = \x^\tau$. Let $\tau \not\in \S$ be inclusion-minimal.  If
    $n \not\in \tau$, then $\tau \not\in \Scl(V^0) \cup \Scl(V^1)$ and there
    are polynomials $f^i_\tau \in I(V^i)$ with leading term $\x^\tau$ for
    $i=0,1$. By virtue of the lexicographic term order, the polynomial $f_\tau
    = (1-x_n)f^0_\tau + x_n f^1_\tau$ also has leading term $\x^\tau$ and
    vanishes on $V$.

    If $n \in \tau$, then $\sigma := \tau \setminus \{n\}$ is contained in,
    say, $\Scl(V^1)$. Now, $\sigma$ cannot be contained in $\Scl(V^0)$ as
    well, as it would imply $\tau \in n \ast (\Scl(V^0) \cap \Scl(V^1))
    \subseteq \S$. Hence, there is a polynomial $f^0_\sigma \in I(V^0)$ with
    leading term $\sigma$. Consequently, the polynomial $f_\tau :=
    (x_n-1)f^0_\sigma$ has leading term $\x^\tau$ and vanishes on $V$.

    Let $I' \subseteq I(V)$ be the ideal generated by the polynomials $f_\tau$
    for $\tau \not\in \S$. Then 
    \[
        |\S| \ \ge \ \dim_\R \R[\x]/I' \ \ge \ \dim_\R \R[\x]/I(V) \ = \ |V|
        \, .
    \]
    On the other hand, we have $|\S| = |\Scl(V^0)| + |\Scl(V^1)| = |V^0| +
    |V^1| = |V|$. This shows that $I' = I(V)$ and proves the claim.
\end{proof}

If $V = \emptyset$, then $\Scl(V) = \emptyset$. On the other hand, if $|V| =
1$, then $\Scl(V) = \{ \emptyset \}$. This gives starting conditions for a
recursive computation of $\Scl(V)$ for general $V \subseteq \ZO^n$.  In this
case, we can rephrase Lemma~\ref{lem:lex} as follows, which also
yields~\eqref{eqn:mapping}.

\begin{corollary}
    Let $V \subseteq \ZO^n$. Then $\Scl(V)$ is the mapping cone for the
    inclusion
    \[
        \Scl(V^0) \cap \Scl(V^1)  \ \hookrightarrow \ \Scl(V^0) \ \cup \
        \Scl(V^1) \, .
    \]
\end{corollary}

Let $M$ be a matroid with groundset $[n]$. If $n$ is neither a loop nor a
coloop, then $M \setminus n$ is the matroid with bases $B \in \Bases{M}$ with
$n \not\in B$. The contraction is the matroid $M / n$ with bases $B \setminus
n$ for $B \in \Bases{M}$ and $n \in B$. It follows that if $V = V_M$ is the
basis configuration of a matroid $M$, then $V^0 = V_{M \backslash n} \times
\{0\}$ and $V^1 = V_{M / n} \times \{1\}$. This shows Theorem~\ref{thm:lex}.
Note that if $n$ is a loop, then $V^1 = \emptyset$ whereas if $n$ is a
coloop, then $V^0 = \emptyset$.

Lemma~\ref{lem:lex} also gives us a way to prove
Corollary~\ref{cor:bijection}.

\begin{proof}[Proof of Corollary~\ref{cor:bijection}]
    For $M = \{\emptyset\} \in \Mclass$, we have $\Bases{M} = \{\emptyset\} =
    \Scl(M)$ and set $\bij_M( \emptyset ) := \emptyset$.

    For a nonempty matroid $M \in \Mclass$ with $m = m(M)$, we assume by
    induction that $\bij_{M\backslash m}$ and $\bij_{M/m}$ are bijections with the desired properties.
    Let $B \in \Bases{M}$.  If $m\not\in B$, then $B \in \Bases{M\backslash
    m}$ and we set  $\bij_M(B) := \bij_{M \backslash m}(B)$.  If $m \in B$,
    then $B \in \Bases{M/m}$ and let $\tau := \bij_{M / m}(B \setminus m)$.
    If $\tau \not\in \Scl(M \setminus m)$, then we set $\bij_M(B) := \tau$.
    Otherwise, we set $\bij_M(B) := \tau \cup \{m\}$.  It follows from
    \Cref{lem:lex} that this is well-defined and a bijection from $\Bases{M}$
    to $\Scl(M)$.  It also follows from \Cref{lem:lex} that $\bij_M$ is the
    unique bijection for $M \in \Mclass$ also satisfying $\bij_M(B) \subseteq
    B$.
\end{proof}

\section{Standard complexes for lattice path matroids}
\label{sec:lattices}

In this section, we discuss the proposed bijection in \Cref{cor:bijection} for the max-minor class class of lattice path matroids.
As we see below, this bijection closely relates to classical bijective combinatorics on lattice paths.

A \Def{lattice path} $C$ from $(0,0)$ to $(d,n-d)$ in $\Z^2$ is a sequence
of~$d$ \Def{east steps}~$\e$ in direction $(1,0)$ and $n-d$ \Def{north
steps}~$\n$ in direction $(0,1)$.  We denote the collection of all such path
by $\latpaths{n,d}$.  A path $C \in \latpaths{n,d}$ may be represented as a
word in $\{\e,\n\}$, and we refer to its $i$-th letter as~$C_i$.  We refer to
the actual line segment of a step in $\mathbb{R}^2$ between its endpoints as
its \Dfn{realization}.  For later reference, we also define a \Dfn{diagonal
step}~$\d$ in direction $(1,1)$ and an \Dfn{empty step}~$\emp$.  

Say that~$U$ is \Def{weakly above}~$L$ for $U,L \in \latpaths{n,d}$ if~$U$
never goes below~$L$.  In terms of words this means that every prefix~$L$
contains at least as many east steps as the corresponding prefix of~$U$.  For
two such paths, let $\paths{U,L} \subseteq \latpaths{n,d}$ be the set of all
lattice paths weakly between~$U$ and~$L$, i.e, weakly below~$U$ and weakly
above~$L$. We refer to~$U$ as the \Dfn{upper boundary} of $\paths{U,L}$ and
to~$L$ as the \Dfn{lower boundary}.

If $n$ and $d$ are fixed, then we may identify $L \in \latpaths{n,d}$ by its
ordered collection of east steps $\east{L} = \{ l_1 < \cdots < l_d\}$.  If $U
\in \latpaths{n,d}$ with  $\east{U} = \{ u_1 < \cdots < u_d\}$, then $U$ is
weakly above~$L$ if and only if
\begin{equation} \label{eq:comp_order}
    l_i \ \le u_i \quad \text{ for } i=1,\dots,d \, .
\end{equation}
We denote this by $\east{L} \leqcomp \east{U}$.

It was observed by Bonin, de Mier, and Noy~\cite{BMN} that the collection of
lattice path between $U$ and $L$ give rise to a matroid. The special case of
Dyck paths was studied by Ardilla~\cite{Ard2003}. Let $U,L \in
\latpaths{n,d}$ with~$U$ weakly above~$L$. Then 
\[
    \bases{U,L} = \bigset{ \east{C} }{ C \in \paths{U,L} } = \bigset{ Z
    \subseteq [n] }{ |Z| = d, \east{L} \leqcomp Z \leqcomp \east{U} } 
\]
is the collection of bases of a matroid $\matroid{U,L}$ on groundset $[n]$,
called a \Def{lattice path matroid}. The collection $\Mclass$ of lattice path
matroids is closed under deletion, contraction, and duality.

We denote the lexicographic standard complex of $\matroid{U,L}$ by 
$\Scl[U,L] := \Scl(\matroid{U,L})$.
Let
\[
  \bij = \bij_{\matroid{U,L}} : \bases{U,L} \to \Scl[U,L]
\]
be the bijection of Corollary~\ref{cor:bijection}. For a lattice path $C \in \paths{U,L}$ we thus have
\[
    \bij( \east{C} ) \ \subseteq \ \east{C} \, .
\]
The goal of this section is to make this selection of east steps explicit by
means of lattice path combinatorics.
Concretely, this means to find a combinatorial statistic~$\stat$ associating to each path $C \in \paths{U,L}$ a subset $\stat(C$) of its east steps $\east{C}$ such that $\bij(\east{C}) = \stat(C)$.

\subsection{Simple version of main theorem}

As a warm-up, we provide a description for the special case of the trivial lower boundary
\[
  \Ltriv = \underbrace{\e\quad\cdots\quad\e}_{d \text{ times}}\ \underbrace{\n\quad\cdots\quad\n}_{n-d \text{ times}} \in \latpaths{n,d}.
\]
For a path~$C \in \latpaths{n,d}$, the statistic $\stat(C)$ is obtained by marking certain east steps as follows: scan through the word of~$C$ from left to right and mark the step $C_i = \e$ if there are as many $\n$'s to the left of position~$i$ as there are unmarked~$\e$'s.
See \Cref{ex:runningextriv} for an example of this marking process.
As indicated in \Cref{fig:runningextriv}, one may interpret the marked and unmarked east steps of a path~$C$ as well graphically by drawing the \Dfn{$\Ltriv$-marking path} of~$C$.
This path, $\mar{\Ltriv}{C}$, of total length~$d$ consists of east steps~$\e$ and diagonal steps~$\d$.
It starts at $(0,0)$ and uses diagonal steps whenever possible without going above the path~$C$, and otherwise uses east steps.
The marked east steps of~$C$ are then those whose realizations are also east steps of the marking path $\mar{\Ltriv}{C}$.

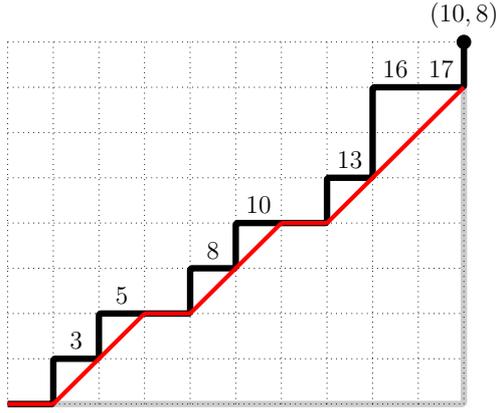
\begin{figure}
  \centering
  \scalebox{0.8}{
    \begin{tikzpicture}[scale=0.75]
      \draw[dotted] (0, 0) grid (10,8);
      \drawPath{0,0,0,0,0,0,0,0,0,0,1,1,1,1,1,1,1,1}{3}{opacity=0.2}{(0,0)}
      \drawPath{0,1,0,1,0,0,1,0,1,0,0,1,0,1,1,0,0,1}{3}{black}{(0,0)}
      \drawPath{0,2,2,0,2,2,0,2,2,2}{2}{red}{(0,0)}
      \node at ( 1.5, 1.4) {$3$};
      \node at ( 2.5, 2.4) {$5$};
      \node at ( 4.5, 3.4) {$8$};
      \node at ( 5.5, 4.4) {$10$};
      \node at ( 7.5, 5.4) {$13$};
      \node at ( 8.5, 7.4) {$16$};
      \node at ( 9.5, 7.4) {$17$};
      \draw[fill=black] (10,8) circle (.15);
      \node[fill=white,inner sep=2pt] at (10,8.6) {$(10,8)$};
    \end{tikzpicture}
  }
  \caption{\label{fig:runningextriv}The path in \Cref{ex:runningextriv} with its $\Ltriv$-marking path.}
\end{figure}

\begin{definition}[Combinatorial statistic, simple version]
  Mark east steps of $C \in \latpaths{n,d}$ using the $\Ltriv$-marking path
  $P = \mar{\Ltriv}{C}$. Then
  \[
    \stat_{\Ltriv}(C) \ := \ \east{C} \setminus \east{P} \ \subseteq \ \east{C}
  \]
  is the \emph{unmarked east steps} in~$C$.
\end{definition}

\begin{example}
\label{ex:runningextriv}
  We mark the path $C \in \latpaths{18,10}$ in \Cref{fig:runningextriv} as
  \[
    \begin{array}{rccccccccccccccccccc}
      C & = & \underline\e & \n & \e & \n & \e & \underline\e & \n & \e & \n & \e & \underline\e & \n & \e & \n & \n & \e & \e & \n \\[5pt]
      \mar{\Ltriv}{C} & = & \e & & \d & &\d & \e & & \d & & \d & \e & & \d & & & \d & \d &
    \end{array}
  \]
  and obtain
  \[
  \stat_\Ltriv(C) = \{ 3,5,8,10,13,16,17 \} \subseteq \{ 1,{\bf 3},{\bf 5},6,{\bf 8},{\bf 10},11,{\bf 13},{\bf 16},{\bf 17} \} = \east{C}.
  \]
\end{example}

Note that $\stat_\Ltriv(C)$ is constructed in a greedy-like manner and, more
over, is independent of the upper boundary $U$.

\begin{remark}
  \label{rem:hooks}
  As observed using the statistics database FindStat~\cite{FindStat}, one may
  as well define $\stat_\Ltriv(C)$ using the following hook placements in the
  area between~$C$ and the lower boundary $\Ltriv$.
  Scan the columns below a given path~$C \in \latpaths{n,d}$ from left to right and place, if possible, the corner box of a south-east hook into the north-most box below~$C$ in that column that is not already covered by other hooks.
  The columns of the placed corner boxes are then exactly the columns that contribute to the statistic.
  We refer to \Cref{fig:hook} for an example of this procedure.
  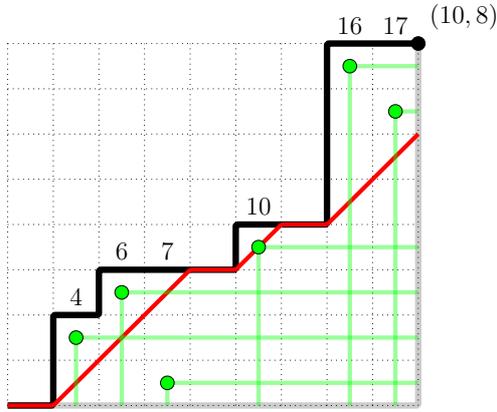
\begin{figure}
    \centering
    \scalebox{0.8}{
      \begin{tikzpicture}[scale=0.75]
        \draw[dotted] (0, 0) grid (9,8);
        \drawPath{0,0,0,0,0,0,0,0,0,1,1,1,1,1,1,1,1}{3}{opacity=0.2}{(0,0)}
        \drawPath{0,1,1,0,1,0,0,0,1,0,0,1,1,1,1,0,0}{3}{black}{(0,0)}
        \drawPath{0,2,2,2,0,2,0,2,2}{2}{red}{(0,0)}
        \node at ( 1.5, 2.4) {$4$};
        \node at ( 2.5, 3.4) {$6$};
        \node at ( 3.5, 3.4) {$7$};
        \node at ( 5.5, 4.4) {$10$};
        \node at ( 7.5, 8.4) {$16$};
        \node at ( 8.5, 8.4) {$17$};
        \draw[fill=black] (9,8) circle (.15);
        \node[fill=white,inner sep=2pt] at (10,8.6) {$(10,8)$};

        \draw[rounded corners=1, line width=2, opacity=0.4, green] (1.5,0) -- (1.5,1.5) -- (9,1.5);
        \draw[rounded corners=1, line width=2, opacity=0.4, green] (2.5,0) -- (2.5,2.5) -- (9,2.5);
        \draw[rounded corners=1, line width=2, opacity=0.4, green] (3.5,0) -- (3.5,0.5) -- (9,0.5);
        \draw[rounded corners=1, line width=2, opacity=0.4, green] (5.5,0) -- (5.5,3.5) -- (9,3.5);
        \draw[rounded corners=1, line width=2, opacity=0.4, green] (7.5,0) -- (7.5,7.5) -- (9,7.5);
        \draw[rounded corners=1, line width=2, opacity=0.4, green] (8.5,0) -- (8.5,6.5) -- (9,6.5);
        \draw[fill=green] (1.5,1.5) circle (.15);
        \draw[fill=green] (2.5,2.5) circle (.15);
        \draw[fill=green] (3.5,0.5) circle (.15);
        \draw[fill=green] (5.5,3.5) circle (.15);
        \draw[fill=green] (7.5,7.5) circle (.15);
        \draw[fill=green] (8.5,6.5) circle (.15);
      \end{tikzpicture}
    }
    \caption{\label{fig:hook}An example of the hook placement of \Cref{rem:hooks}.}
  \end{figure}
\end{remark}

\begin{theorem}[Simple version]
\label{thm:mainsimple}
  Let $U \in \latpaths{n,d}$ and let $\bij = \bij_{M[U,\Ltriv]}$
  Then
  \[
    \stat(C) = \bij(\east{C})\quad \text{ for all } C \in \paths{U,\Ltriv}.
  \]
  In particular, the standard complex in this case can be described as
  \[
    \Scl(\matroid{U,\Ltriv}) = \bigset{ \stat(C) }{ C \in \paths{U,\Ltriv}}.
  \]
\end{theorem}

\subsection{General version of main theorem}

After having given the simple version of the statistic for the trivial lower
boundary~$\Ltriv$, we modify the definition of the marking path depending on
the
given lower boundary by introducing another intermediate path.
We remark that it seems to be not possible to give a description in terms of hook placements as in \Cref{rem:hooks} for this general situation of non-trivial lower bounds.

Let $C,L \in \latpaths{n,d}$ with~$C$ weakly above~$L$.
The \Dfn{$L$-demarcation path} of~$C$, $\dem{L}{C}$, is defined to start at $(0,0)$ with steps
\[
  \dem{L}{C}_i =
  \begin{cases}
    \n & \text{ if }   C_i = \n,\ L_i = \n \\
    \e & \text{ if }   C_i = \e,\ L_i = \e \\
    \d & \text{ if }   C_i = \n,\ L_i = \e \\
    \emp & \text{ if } C_i = \e,\ L_i = \n
  \end{cases}\ .
\]
Graphically, this means that $\dem{L}{C}$ starts at $(0,0)$ and then is obtained by combining the $\e$-coordinate of~$L$ with the $\n$-coordinate of~$C$.
In particular, this implies that the $L$-demarcation path of~$C$ is weakly between~$L$ and~$C$.

\begin{example}
\label{ex:runningexhard}
  The $L$-demarcation path $\dem{L}{C}$ of the path $C \in \latpaths{21,11}$ in \Cref{fig:runningexhard} relative to the given lower boundary $L \in \latpaths{21,11}$ is:
  \[
    \begin{array}{rcccccccccccccccccccccc}
      C & = & \n & \n & \n & \e & \e & \e & \n & \e & \e & \n & \e & \n & \n & \e & \n & \n & \n & \e & \e & \e & \e
      \\
      L & = & \e & \e & \e & \n & \e & \e & \e & \n & \n & \e & \n & \e & \e & \n & \n & \n & \n & \e & \e & \n & \n
      \\[10pt]
      \dem{L}{C} & = & \d & \d & \d & \emp & \e & \e & \d & \emp & \emp & \d & \emp & \d & \d & \emp & \n & \n & \n & \e & \e & \emp & \emp
    \end{array}
  \]
\end{example}
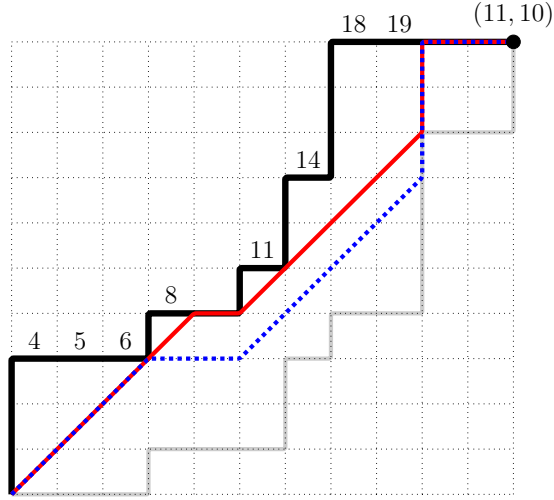
\begin{figure}
  \centering
  \scalebox{0.8}{
    \begin{tikzpicture}[scale=0.75]
      \draw[dotted] (0, 0) grid (11,10);
      \drawPath{0,0,0,1,0,0,0,1,1,0,1,0,0,1,1,1,1,0,0,1,1}{2}{opacity=0.2}{(0,0)}
      \drawPath{1,1,1,0,0,0,1,0,0,1,0,1,1,0,1,1,1,0,0,0,0}{3}{black}{(0,0)}
      \drawPath{2,2,2,2,0,2,2,2,2,1,1,0,0}{2}{red}{(0,0)}
      \drawPath{2,2,2,0,0,2,2,2,2,1,1,1,0,0}{2}{blue, dotted}{(0,0)}
      \node at ( 0.5, 3.4) {$4$};
      \node at ( 1.5, 3.4) {$5$};
      \node at ( 2.5, 3.4) {$6$};
      \node at ( 3.5, 4.4) {$8$};
      \node at ( 5.5, 5.4) {$11$};
      \node at ( 6.5, 7.4) {$14$};
      \node at ( 7.5,10.4) {$18$};
      \node at ( 8.5,10.4) {$19$};
      \draw[fill=black] (11,10) circle (.15);
      \node[fill=white,inner sep=2pt] at (11,10.6) {$(11,10)$};
    \end{tikzpicture}
  }
  \caption{The path in \Cref{ex:runningexhard} with the given lower boundary $L$, its $L$-demarcation path in dotted blue and its $L$-marking path in red.}
  \label{fig:runningexhard}
\end{figure}

Let $C,L \in \latpaths{n,d}$ with~$C$ weakly above~$L$.
The \Dfn{$L$-marking path} of~$C$, $\mar{L}{C}$, consists of diagonal steps $\d$, east steps~$\e$ and north steps $\n$.
It goes from $(0,0)$ to $(d,n-d)$ and uses diagonal steps whenever possible without going above the path~$C$ or below the $L$-demarcation path $\dem{L}{C}$.
If it would go above~$C$, it uses an east steps instead, and if it would go below $\dem{L}{C}$, it uses a north step instead.
The marked east steps of~$C$ are then those whose realizations are also east steps of the marking path $\mar{L}{C}$.

\begin{definition}[Combinatorial statistic, general version]
    Let $C,L \in \latpaths{n,d}$ with~$C$ weakly above~$L$ and mark east steps
    of~$C$ using the $L$-marking path $P = \mar{L}{C}$.  Then
  \[
    \stat_{L}(C) \ := \ \east{C} \setminus \east{P} \ \subseteq \ \east{C}
  \]
  is the \emph{unmarked east steps} in~$C$.
\end{definition}

\begin{example}
\label{ex:runningexhard2}
  We mark east steps in the path $C \in \latpaths{21,11}$ in \Cref{fig:runningexhard} relative to the given lower boundary~$L \in \latpaths{21,11}$ as
  \[
    C = \n\ \n\ \n\ \e\ \e\ \e\ \n\ \e\ \underline{\e}\ \n\ \e\ \n\ \n\ \e\ \n\ \n\ \n\ \e\ \e\ \underline{\e}\ \underline{\e}
  \]
  and obtain
  \[
  \stat_L(C) = \{ 4,5,6,8,11,14,18,19 \} \subseteq \{ {\bf 4},{\bf 5},{\bf 6},{\bf 8},9,{\bf 11},{\bf 14},{\bf 18},{\bf 19},20,21 \} = \east{C}.
  \]
\end{example}

\begin{theorem}[General version]
\label{thm:mainhard}
  Let $U,L \in \latpaths{n,d}$ with~$U$ weakly above~$L$ and let $\bij = \bij_{M[U,L]}$.
  Then
  \[
    \stat(C) = \bij(\east{C})\quad \text{ for all } C \in \paths{U,L}.
  \]
  In particular, the standard complex in this case can be described as
  \[
    \Scl(\matroid{U,L}) = \bigset{ \stat(C) }{ C \in \paths{U,L}}.
  \]
\end{theorem}

Surprisingly, this combinatorial description yields the following property of the bijection~$\bij$ in the case of lattice path matroids.

\newcommand\restr[2]{{%
    \left.\kern-\nulldelimiterspace %
    #1 %
    \vphantom{\big|} %
    \right|_{#2} %
  }}

\begin{corollary}
\label{cor:mainhard}
  Let $U,U',L \in \latpaths{n,d}$ such that~$U$ is weakly above~$U'$ and~$U'$ is weakly above~$L$.
  Then
  \[
    \restr{\bij_{\matroid{U,L}}}{\bases{U',L}} = \bij_{\matroid{U',L}}.
  \]
  In particular, $\Scl[U',L] \subseteq \Scl[U,L]$.
\end{corollary}

\begin{proof}
  This is a direct consequence of \Cref{thm:mainhard}.
\end{proof}

In light of \Cref{prop:dual}, this theorem also yields the following open problem.
For a path $C \in \latpaths{n,d}$ we write $C^* \in \latpaths{n,n-d}$ to be the path given by $C^*_i = \e$ if $C_i = \n$ and vice versa.

\begin{openproblem}
  For $U,L \in \latpaths{n,d}$ with~$U$ weakly above~$L$, find a bijection
  \[
    \varphi : \paths{U,L}\ \tilde \longleftrightarrow\ \paths{L^*,U^*}
  \]
  with the property that $\stat_L(C) = \stat_{L^*}(\varphi(C))$ for all $C \in \paths{U,L}$.
\end{openproblem}

Given $J \in \Scl(M[U,L]) = \Scl(M[L^*,U^*])$, we will see below how to construct the unique $C \in \paths{U,L}$ and the unique $C' \in \paths{L^*,U^*}$ with $\stat_L(C) = \stat_{U^*}(C') = J$.
The open problem is thus to provide a bijection without referring to the set~$J$ in the first place.
By definition, the bijection in question must send~$L$ to~$U^*$ and~$U$ to~$L^*$.
In particular, it is not the obvious bijection $C \mapsto C^*$.

\subsection{Proof of \texorpdfstring{\Cref{thm:mainsimple,thm:Lrecurrencesimple}}{combinatorial theorems}}

Let $U,L \in \latpaths{n,d}$ with~$U$ weakly above~$L$ and consider the obvious decomposition
\[
  \paths{U,L} = \paths{U,L}^\e \sqcup \paths{U,L}^\n
\]
where $\paths{U,L}^\e$ and $\paths{U,L}^\n$ denote all paths in $\paths{U,L}$ ending in an east step or, respectively, in a north step.
Observe that
\[
  \paths{U,L}^\e \neq \emptyset \Leftrightarrow U_n = \e, \quad 
  \paths{U,L}^\n \neq \emptyset \Leftrightarrow L_n = \n.
\]
We also write $\paths{U,L}^\e_\circ$ and $\paths{U,L}^\n_\circ$ for the
corresponding paths with the last east or, respectively, north step removed.
If they exist, these truncated paths may clearly again be realized as lattice
paths between boundaries:
\begin{align*}
  \paths{U,L}^\e_\circ &= \paths{U^\e_\circ,L^\e_\circ} \quad\text{ if } U_n = \e \\
  \paths{U,L}^\n_\circ &= \paths{U^\n_\circ,L^\n_\circ} \quad\text{ if } L_n = \n
\end{align*}
where~$U^\e_\circ,L^\e_\circ \in \latpaths{n-1,d-1}$ and $U^\n_\circ,L^\n_\circ \in \latpaths{n-1,d}$ are obtained from~$U$ and~$L$ by respectively removing the last east or north step.
(Note that in the first situation, the last east step of~$L$ might not be its last step and in the second, the last north step of~$U$ might not be its last step.)

\subsubsection*{Proof of \texorpdfstring{\Cref{thm:mainsimple}}{main theorem (simple version)}}

We first prove that the simple version of the combinatorial statistic by showing the following theorem.
\Cref{thm:mainsimple} then follows with \Cref{thm:lex}.

\begin{theorem}
\label{thm:Lrecurrencesimple}
  Let $U \in \latpaths{n,d}$ and set
  \begin{align*}
    \mathcal{S}_\mathcal{L} &= \bigset{\stat_\Ltriv(C)}{C \in \paths{U,\Ltriv}\hspace*{4.5pt}}, \\
    \mathcal{S}^\e_\mathcal{L} &= \bigset{\stat_\Ltriv(C)}{C \in \paths{U,\Ltriv}^\e_\circ}, \\
    \mathcal{S}^\n_\mathcal{L} &= \bigset{\stat_\Ltriv(C)}{C \in \paths{U,\Ltriv}^\n_\circ}.
    \intertext{We then have}
    \mathcal{S}_\mathcal{L} &= \mathcal{S}^\e_\mathcal{L}\ \cup\  \mathcal{S}^\n_\mathcal{L}\  \cup\  n * \big(\mathcal{S}^\e_\mathcal{L} \cap \mathcal{S}^n_\mathcal{L}\big). 
  \end{align*}
\end{theorem}

We prove this theorem in several steps and refer to \Cref{fig:pathswithsamestat} for graphical illustrations of the given arguments.

\begin{lemma}
  \label{lem:neccondsimple}
  Let $C \in \latpaths{n,d}$ with $\stat_\Ltriv(C) = \{j_1 < \dots < j_\ell\}$.
  Then
  \[
    \ell \leq d \leq n-\ell\quad \text{and}\quad j_k \geq 2k\text{ for all }1 \leq k \leq \ell.
  \]
\end{lemma}

\begin{proof}
  We have $\stat_\Ltriv(C) \subseteq \east{C}$ and therefore $\ell = |\stat_\Ltriv(C)| \leq |\east{C}| = d$.
  Moreover, the marking path $\mar{\Ltriv}{C}$ consists of $\ell$ diagonal steps and of $d-\ell$ east steps and thus ends at the point $(d,\ell)$.
  Because $C$ ends at $(d,n-d)$ and is weakly above $\mar{\Ltriv}{C}$, we also have $\ell \leq n-d$ or, equivalently, $d \leq n-\ell$.

  The second property $j_k \geq 2k$ follows from the observation that, if there was a smallest~$k$ with $j_k < 2k$, the realization of the $j_k$-th step (which is an east step) in the path~$C$ would also be an east step of its marking path~$\mar{\Ltriv}{C}$ and thus $j_k \notin \stat_\Ltriv(C)$.
\end{proof}

\begin{lemma}
  \label{lem:suffcondsimple}
  Fix~$n,d$ and $J \subseteq \{1,\dots,n\}$ with $|J| = \ell$ such that $\ell \leq d < n-\ell$.
  Then there is an explicit bijection
  \[
  \bigset{C \in \latpaths{n,d}}{\stat_\Ltriv(C) = J} \
  \tilde\longleftrightarrow \ \bigset{C' \in
  \latpaths{n,d+1}}{\stat_\Ltriv(C') = J}.
  \]
\end{lemma}

\begin{proof}
    Let $C \in \latpaths{n,d}$ such that $|\stat_\Ltriv(C)| = \ell$.  As we
    have already seen in the proof of \Cref{lem:neccondsimple}, the marking
    path of~$C$ ends at $(d,\ell)$.  Because~$C$ ends at $(d,n-d)$ and $\ell <
    n-d$, the path~$C$ contains a unique last north step whose realization
    starts on the marking path.  Replacing this north step by an east step
    yields a path~$C' \in \latpaths{n,d+1}$ for which $\stat_\Ltriv(C') =
    \stat_\Ltriv(C)$.

    On the other hand, let $C' \in \latpaths{n,d+1}$ such that
    $|\stat_\Ltriv(C')| = \ell$.  Because $d+1 > \ell$, the path $C'$ contains
    a last east step whose realization is also an east step of its marking
    path $\mar{\Ltriv}{C'}$.  Replacing this east step by a north step yields
    a path~$C \in \latpaths{n,d}$ with $\stat_\Ltriv(C) = \stat_\Ltriv(C')$.

    Those two operations are inverses of each other and are thus the desired
    bijection.  This can be seen as follows: Consider $C$ and its marking path
    $P = \mar{\Ltriv}{C}$, and also $C'$ and its marking path $P' =
    \mar{\Ltriv}{C'}$, and let~$i$ be the index of the north step in~$C$ that
    is replaced in $C'$ by an east step (or vice versa).  Then the relative
    positions of~$C$,~$P$ and, respectively, the relative positions
    of~$C'$,~$P'$ coincide, except that $C'$ is horizontally one step closer
    to~$P'$ than~$C$ is to~$P$.  We refer to \Cref{fig:pathswithsamestat} for
    several examples.
\end{proof}

\begin{proposition}
\label{prop:constructpathfromstatsimple}
  Fix~$n,d$ and $J = \{j_1 < \cdots < j_\ell\} \subseteq \{1,\dots,n\}$.
  Then there exists a path~$C \in \latpaths{n,d}$ with $\stat_\Ltriv(C) = J$ if and only if the two conditions
  \[
    j_k \geq 2k\text{ for all }1 \leq k \leq \ell \quad\text{and}\quad
    \ell \leq d \leq n-\ell
  \]
  are both satisfied.
  Moreover, the path with this property is unique in this case.
\end{proposition}

\begin{proof}
  We have seen in \Cref{lem:neccondsimple} that these two conditions are both necessary for such a path to exist.
  For $d = \ell$ the (unique) path~$C$ with $\east{C} = \{j_1,\dots,j_\ell\}$ has the desired property $\stat_\Ltriv(C) = \east{C}$.
  It then follows with \Cref{lem:suffcondsimple} that there exists, for any~$d$ with $\ell \leq d \leq n-\ell$, a unique path~$C'$ in $\latpaths{n,d}$ with $\stat_\Ltriv(C') = J$.
  (This path is obtained from $C \in \latpaths{n,\ell}$ by multiple replacements of north steps by east steps as described.)
\end{proof}

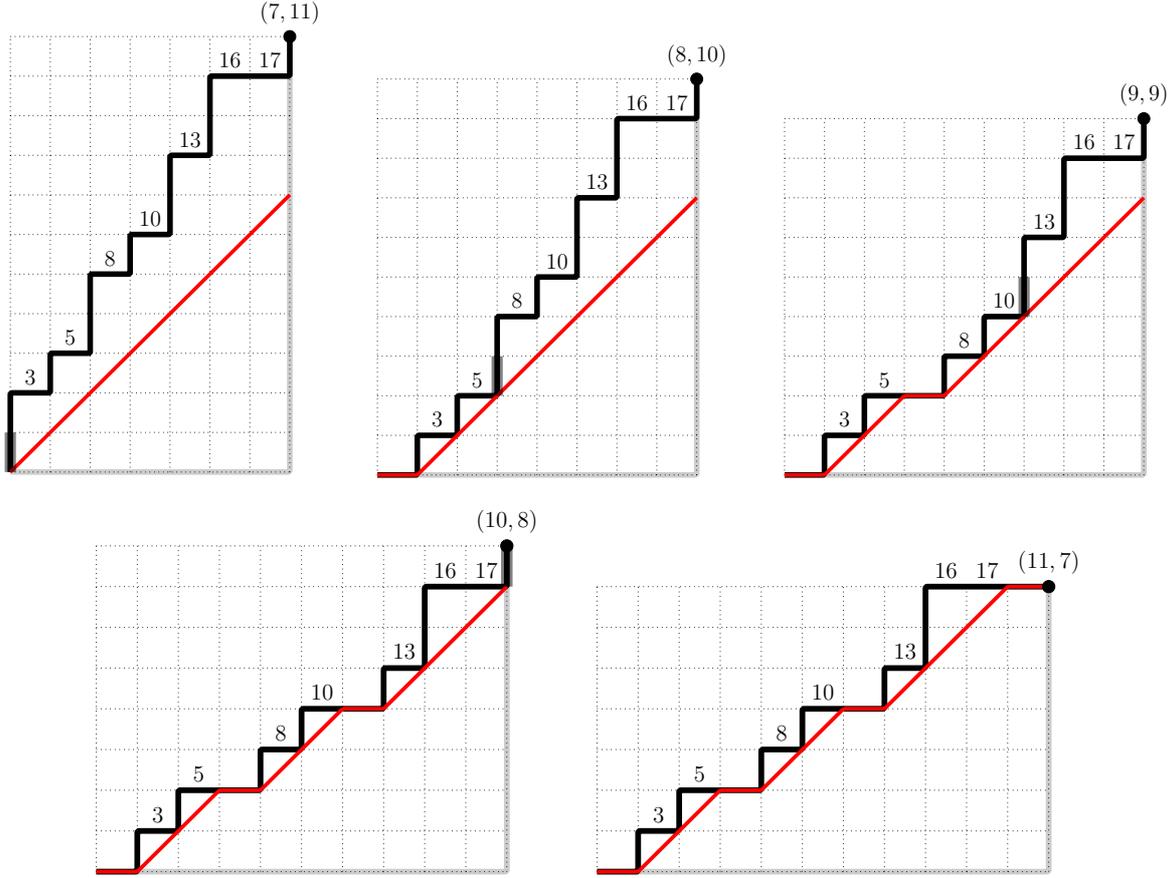
\begin{figure}
  \centering
  \scalebox{0.7}{
    \begin{tikzpicture}[scale=0.75]
      \draw[dotted] (0, 0) grid (7,11);
      \drawPath{0,0,0,0,0,0,0,1,1,1,1,1,1,1,1,1,1,1}{3}{opacity=0.2}{(0,0)}
      \drawPath{1,1,0,1,0,1,1,0,1,0,1,1,0,1,1,0,0,1}{3}{black}{(0,0)}
      \drawPath{2,2,2,2,2,2,2}{2}{red}{(0,0)}
      \node at ( 0.5, 2.4) {$3$};
      \node at ( 1.5, 3.4) {$5$};
      \node at ( 2.5, 5.4) {$8$};
      \node at ( 3.5, 6.4) {$10$};
      \node at ( 4.5, 8.4) {$13$};
      \node at ( 5.5,10.4) {$16$};
      \node at ( 6.5,10.4) {$17$};
      \draw[line width=6, opacity=0.5] (0, 0) -- (0,1);
      \draw[fill=black] (7,11) circle (.15);
      \node[fill=white,inner sep=2pt] at ( 7,11.6) {$(7,11)$};
    \end{tikzpicture}

    \qquad

    \begin{tikzpicture}[scale=0.75]
      \draw[dotted] (0, 0) grid (8,10);
      \drawPath{0,0,0,0,0,0,0,0,1,1,1,1,1,1,1,1,1,1}{3}{opacity=0.2}{(0,0)}
      \drawPath{0,1,0,1,0,1,1,0,1,0,1,1,0,1,1,0,0,1}{3}{black}{(0,0)}
      \drawPath{0,2,2,2,2,2,2,2}{2}{red}{(0,0)}
      \node at ( 1.5, 1.4) {$3$};
      \node at ( 2.5, 2.4) {$5$};
      \node at ( 3.5, 4.4) {$8$};
      \node at ( 4.5, 5.4) {$10$};
      \node at ( 5.5, 7.4) {$13$};
      \node at ( 6.5, 9.4) {$16$};
      \node at ( 7.5, 9.4) {$17$};
      \draw[line width=6, opacity=0.5] (3,2) -- (3,3);
      \draw[fill=black] (8,10) circle (.15);
      \node[fill=white,inner sep=2pt] at ( 8,10.6) {$(8,10)$};
    \end{tikzpicture}

    \qquad

    \begin{tikzpicture}[scale=0.75]
      \draw[dotted] (0, 0) grid (9,9);
      \drawPath{0,0,0,0,0,0,0,0,0,1,1,1,1,1,1,1,1,1}{3}{opacity=0.2}{(0,0)}
      \drawPath{0,1,0,1,0,0,1,0,1,0,1,1,0,1,1,0,0,1}{3}{black}{(0,0)}
      \drawPath{0,2,2,0,2,2,2,2,2}{2}{red}{(0,0)}
      \node at ( 1.5, 1.4) {$3$};
      \node at ( 2.5, 2.4) {$5$};
      \node at ( 4.5, 3.4) {$8$};
      \node at ( 5.5, 4.4) {$10$};
      \node at ( 6.5, 6.4) {$13$};
      \node at ( 7.5, 8.4) {$16$};
      \node at ( 8.5, 8.4) {$17$};
      \draw[line width=6, opacity=0.5] (6,4) -- (6,5);
      \draw[fill=black] (9,9) circle (.15);
      \node[fill=white,inner sep=2pt] at (9,9.6) {$(9,9)$};
    \end{tikzpicture}
  }

  \bigskip

  \scalebox{0.72}{
    \begin{tikzpicture}[scale=0.75]
      \draw[dotted] (0, 0) grid (10,8);
      \drawPath{0,0,0,0,0,0,0,0,0,0,1,1,1,1,1,1,1,1}{3}{opacity=0.2}{(0,0)}
      \drawPath{0,1,0,1,0,0,1,0,1,0,0,1,0,1,1,0,0,1}{3}{black}{(0,0)}
      \drawPath{0,2,2,0,2,2,0,2,2,2}{2}{red}{(0,0)}
      \node at ( 1.5, 1.4) {$3$};
      \node at ( 2.5, 2.4) {$5$};
      \node at ( 4.5, 3.4) {$8$};
      \node at ( 5.5, 4.4) {$10$};
      \node at ( 7.5, 5.4) {$13$};
      \node at ( 8.5, 7.4) {$16$};
      \node at ( 9.5, 7.4) {$17$};
      \draw[line width=6, opacity=0.5] (10,7) -- (10,8);
      \draw[fill=black] (10,8) circle (.15);
      \node[fill=white,inner sep=2pt] at (10,8.6) {$(10,8)$};
    \end{tikzpicture}

    \qquad

    \begin{tikzpicture}[scale=0.75]
      \draw[dotted] (0, 0) grid (11,7);
      \drawPath{0,0,0,0,0,0,0,0,0,0,0,1,1,1,1,1,1,1}{3}{opacity=0.2}{(0,0)}
      \drawPath{0,1,0,1,0,0,1,0,1,0,0,1,0,1,1,0,0,0}{3}{black}{(0,0)}
      \drawPath{0,2,2,0,2,2,0,2,2,2,0}{2}{red}{(0,0)}
      \node at ( 1.5, 1.4) {$3$};
      \node at ( 2.5, 2.4) {$5$};
      \node at ( 4.5, 3.4) {$8$};
      \node at ( 5.5, 4.4) {$10$};
      \node at ( 7.5, 5.4) {$13$};
      \node at ( 8.5, 7.4) {$16$};
      \node at ( 9.5, 7.4) {$17$};
      \draw[fill=black] (11,7) circle (.15);
      \node[fill=white,inner sep=2pt] at (11,7.6) {$(11,7)$};
    \end{tikzpicture}
  }
  \caption{All paths~$C \in \latpaths{18,d}$ with $\stat_\Ltriv(C) = \{ 3,5,8,10,13,16,17\}$, one for each $d \in \{ 7, \dots, 11\}$.}
  \label{fig:pathswithsamestat}
\end{figure}

Let $U \in \latpaths{n,d}$.
For the following treatment, we set
\begin{equation}
\begin{aligned}
  \paths{U,\Ltriv}^{\e} &= \paths{U,\Ltriv}^{\e,-} \sqcup \paths{U,\Ltriv}^{\e,+} \\
  \paths{U,\Ltriv}^{\e}_\circ &= \paths{U,\Ltriv}^{\e,-}_\circ \sqcup \paths{U,\Ltriv}^{\e,+}_\circ
\end{aligned}
\label{eq:Ltrive}
\end{equation}
where $\paths{U,\Ltriv}^{\e,-}$ are the paths
that go weakly below the diagonal through the point $(n-d,d)$ and
$\paths{U,\Ltriv}^{\e,+}$ is its complement in $\paths{U,\Ltriv}^{\e}$. The sets
$\paths{U,\Ltriv}^{\e,-}_\circ$ and $\paths{U,\Ltriv}^{\e,+}_\circ$ are
obtained by removing the last east step.

\begin{corollary}
\label{prop:intersectionsimple}
  Let $U \in \latpaths{n,d}$ with $U_n = \e$.
  Then
  \[
  \bigset{\stat_\Ltriv(C)}{ C \in \paths{U,\Ltriv}^\n_\circ} = \bigset{\stat_\Ltriv(C)}{ C \in \paths{U,\Ltriv}^{\e,-}_\circ}.
  \]
\end{corollary}

\begin{proof}
  Let $C \in \paths{U,\Ltriv}^\e_\circ \subseteq \latpaths{n-1,d-1}$.
  Such a path exists by the assumption $U_n = \e$.
  \Cref{prop:constructpathfromstatsimple} implies that there exists a (then unique) path in $\latpaths{n-1,d}$ with the same statistic value as~$C$ if and only if $d \leq n-1-\ell$.
  Let~$k$ be the number of marked east steps in~$C$, this is, $k = d-1 - \ell$.
  Then $d \leq n-1-\ell$ if and only if $k \geq 2d-n$.
  This is the case if and only if $C \in \paths{U,\Ltriv}^{\e,-}_\circ$.
\end{proof}

\begin{proof}[Proof of \Cref{thm:Lrecurrencesimple}]
  Consider the disjoint decomposition
  \[
    \paths{U,\Ltriv} = \paths{U,\Ltriv}^{\e,-} \cup \paths{U,\Ltriv}^{\e,+} \cup \paths{U,\Ltriv}^{\n}\ .
  \]
  Observe that the definition of the statistic implies for a path $C \in \paths{U,\Ltriv}$ that
  \[
    n \in \stat_\Ltriv(C)\ \Longleftrightarrow\ C \in \paths{U,\Ltriv}^{\e,-}\ .
  \]
  If $U_n = \n$, we have $\Se_\mathcal{L} = \emptyset$ and thus $S_\mathcal{L} = \Sn_\mathcal{L}$.
  Otherwise, we have $U_n = \e$ and obtain
  \begin{align*}
    \bigset{\stat_\Ltriv(C)}{C \in \paths{U,\Ltriv}} =\ 
      &
      \bigset{ \stat_\Ltriv(C) \cup \{n\} }{ C \in \paths{U,\Ltriv}^{\e,-}_\circ}\ \cup
      \\
      &
      \bigset{ \stat_\Ltriv(C) \hspace*{32.5pt} }{ C  \in \paths{U,\Ltriv}^{\e,+}_\circ }\ \cup
      \\
      &
      \bigset{ \stat_\Ltriv(C) \hspace*{32.5pt} }{ C \in \paths{U,\Ltriv}^{\n}_\circ }.
  \end{align*}
  \Cref{prop:intersectionsimple} now implies \Cref{thm:Lrecurrencesimple}.
\end{proof}

\subsection*{Proof of \texorpdfstring{\Cref{thm:mainhard}}{main theorem (general version)}}

We next generalize and modify the arguments for the simple version to obtain the general version.
We show the following theorem and \Cref{thm:mainhard} then follows with \Cref{thm:lex}.

\begin{theorem}
\label{thm:Lrecurrencehard}
  Let $U,L \in \latpaths{n,d}$ with~$U$ weakly above~$L$ and set
  \begin{align*}
    \mathcal{S}_\mathcal{L} &= \bigset{\stat_L(C)}{C \in \paths{U,L}\hspace*{4.5pt}}, \\
    \mathcal{S}^\e_\mathcal{L} &= \bigset{\stat_L(C)}{C \in \paths{U,L}^\e_\circ}, \\
    \mathcal{S}^\n_\mathcal{L} &= \bigset{\stat_L(C)}{C \in \paths{U,L}^\n_\circ}.
    \intertext{We then have}
    \mathcal{S}_\mathcal{L} &= \mathcal{S}^\e_\mathcal{L}\ \cup\  \mathcal{S}^\n_\mathcal{L}\  \cup\  n * \big(\mathcal{S}^\e_\mathcal{L} \cap \mathcal{S}^n_\mathcal{L}\big). 
  \end{align*}
\end{theorem}

We first adapt the definition of the demarcation path and the marking path to the slightly more general setting of $C \in \latpaths{n,d}$ and $L \in \latpaths{n,d'}$ with $d \leq d'$.
The definition of~$C$ being weakly above~$L$ generalizes verbatim and $d \leq d'$ is a necessary condition in this case.
One may then extend~$C$ and~$L$ by appending $d'-d$ many east steps to~$C$ and $d'-d$ many north steps to~$L$ to obtain two paths in $\latpaths{n+d'-d,d'}$.
The definitions of the $L$-demarcation path and the $L$-marking path of~$C$ are given by using these extensions.
\Cref{fig:modificationlemma} shows one such example.

Unfortunately, we do not have an analogue of \Cref{lem:neccondsimple} in the general situation.
We nonetheless still have the following lemma.

\begin{lemma}
  \label{lem:suffcondhard}
  Fix~$n,d<d'$ and $J \subseteq \{1,\dots,n\}$ with $|J| = \ell \leq d$ and $L \in \latpaths{n,d'}$.
  Then there is an explicit bijection
  \[
    \bigset{C \in \latpaths{n,d}}{ \begin{array}{c}
        C \text{ weakly above }L, \\
        \stat_ L(C) = J, \\
        (d,n-d) \notin \mar{L}{C}
                                    \end{array}
    }
    \tilde\longleftrightarrow
    \bigset{C' \in \latpaths{n,d+1}}{ \begin{array}{c}
        C' \text{ weakly above }L, \\
        \stat_L(C') = J
      \end{array}
    } .
  \]
\end{lemma}

In the lemma, we mean by $(d,n-d) \notin \mar{L}{C}$ that the marking path $\mar{L}{C}$ of~$C$ does not go through the endpoint $(d,n-d)$ of~$C$.

\begin{proof}
  The proof of \Cref{lem:suffcondsimple} generalizes verbatim:
  Let $C \in \latpaths{n,d}$ be weakly above~$L$ such that $|\stat_L(C)| = \ell$.
  Because $(d,n-d) \notin \mar{L}{C}$, the path~$C$ contains a unique last north step whose realization start on the marking path $\mar{L}{C}$.
  Replacing this north step by an east step yields a path~$C' \in \latpaths{n,d+1}$ with $\stat_L(C') = \stat_L(C)$ which is also weakly above~$L$.
  
  On the other hand, let $C' \in \latpaths{n,d+1}$ be weakly above~$L$ such that $|\stat_L(C')| = \ell$.
  Because $d+1 > \ell$, the path $C'$ contains a last east step whose realization is also an east step of its marking path $\mar{L}{C'}$.
  Replacing this east step by a north step yields a path~$C \in \latpaths{n,d}$ with $\stat_L(C) = \stat_L(C')$ and $(d,n-d) \notin \mar{L}{C}$.

  Those two operations are inverses of each other and are thus the desired bijection.
  This can be seen as follows:
  Consider $C$, $D = \dem{L}{C}$ and $M = \mar{L}{C}$, and also $C'$, $D' = \dem{L}{C'}$ and $M' = \mar{L}{C'}$, and let~$i$ be the index of the north step in~$C$ that is replaced in $C'$ by an east step (or vice versa).
  Then the relative positions of~$C$,~$D$,~$M$ and, respectively, the relative positions of~$C'$,~$D'$,~$M'$ coincide, except that $C'$ is horizontally one step closer to~$D'$ and to~$M'$ than $C$ is to~$D$ and to~$M$.
  We refer to \Cref{fig:modificationlemma} for two examples.
\end{proof}

\begin{figure}
  \centering
  \scalebox{0.7}{
    \begin{tikzpicture}[scale=0.75]
      \draw[dotted] (0, 0) grid (11,13);
      \drawPath{0,0,0,1,0,0,0,1,1,0,1,0,0,1,1,1,1,0,0,1,1}{2}{opacity=0.2}{(0,0)}
      \drawPath{1,1,1,0,0,0,1,0,1,1,0,1,1,0,1,1,1,0,0,1,1}{3}{black}{(0,0)}
      \drawPath{2,2,2,0,0,2,1,2,2,2,1,1,1,0,0,1,1}{2}{blue, dotted}{(0,0)}
      \drawPath{2,2,2,2,2,2,2,2,2,1,1,2,2}{2}{red}{(0,0)}
      \node at ( 0.5, 3.4) {$4$};
      \node at ( 1.5, 3.4) {$5$};
      \node at ( 2.5, 3.4) {$6$};
      \node at ( 3.5, 4.4) {$8$};
      \node at ( 4.5, 6.4) {$11$};
      \node at ( 5.5, 8.4) {$14$};
      \node at ( 6.5,11.4) {$18$};
      \node at ( 7.5,11.4) {$19$};
      \draw[line width=6, opacity=0.5] (4,4) -- (4,5);
      \draw[fill=black] (4, 4) circle (.15);
      \draw[very thick] (12,6) edge[->] (12.8,6);
      \draw[fill=black] (8,13) circle (.15);
      \draw[fill=black, opacity=0.5] (11,10) circle (.15);
      \node at (13.3,6) {};
    \end{tikzpicture}
    \begin{tikzpicture}[scale=0.75]
      \draw[dotted] (0, 0) grid (11,12);
      \drawPath{0,0,0,1,0,0,0,1,1,0,1,0,0,1,1,1,1,0,0,1,1}{2}{opacity=0.2}{(0,0)}
      \drawPath{1,1,1,0,0,0,1,0,0,1,0,1,1,0,1,1,1,0,0,1,1}{3}{black}{(0,0)}
      \drawPath{2,2,2,0,0,2,2,2,2,1,1,1,0,0,1,1}{2}{blue, dotted}{(0,0)}
      \drawPath{2,2,2,2,0,2,2,2,2,1,1,2,2}{2}{red}{(0,0)}
      \node at ( 0.5, 3.4) {$4$};
      \node at ( 1.5, 3.4) {$5$};
      \node at ( 2.5, 3.4) {$6$};
      \node at ( 3.5, 4.4) {$8$};
      \node at ( 5.5, 5.4) {$11$};
      \node at ( 6.5, 7.4) {$14$};
      \node at ( 7.5,10.4) {$18$};
      \node at ( 8.5,10.4) {$19$};
      \draw[line width=6, opacity=0.5] (4,4) -- (5,4);
      \draw[fill=black] (4, 4) circle (.15);
      \draw[fill=black] (9,12) circle (.15);
      \draw[fill=black, opacity=0.5] (11,10) circle (.15);
    \end{tikzpicture}
  }

  \vspace{20pt}

  \scalebox{0.7}{
    \begin{tikzpicture}[scale=0.75]
      \draw[dotted] (0, 0) grid (10,10);
      \drawPath{0,0,1,1,0,1,0,0,0,1,0,0,0,0,1,1}{2}{opacity=0.2}{(0,0)}
      \drawPath{1,1,0,1,1,1,0,1,1,0,0,1,1,1,0,0}{3}{black}{(0,0)}
      \drawPath{2,2,1,2,1,0,2,2,0,2,2,2}{2}{blue, dotted}{(0,0)}
      \drawPath{2,2,1,2,1,2,2,2,2,2,0,0}{2}{red}{(0,0)}
      \node at ( 0.5, 2.4) {$3$};
      \node at ( 1.5, 5.4) {$7$};
      \node at ( 2.5, 7.4) {$10$};
      \node at ( 3.5, 7.4) {$11$};
      \node at ( 4.5,10.4) {$15$};
      \node at ( 5.5,10.4) {$16$};
      \draw[line width=6, opacity=0.5] (0,0) -- (0,1);
      \draw[fill=black] (0, 0) circle (.15);
      \draw[very thick] (11,5) edge[->] (11.8,5);
      \draw[fill=black] (6,10) circle (.15);
      \draw[fill=black, opacity=0.5] (10,6) circle (.15);
      \node at (12.3,6) {};
    \end{tikzpicture}
    \begin{tikzpicture}[scale=0.75]
      \draw[dotted] (0, 0) grid (10,9);
      \drawPath{0,0,1,1,0,1,0,0,0,1,0,0,0,0,1,1}{2}{opacity=0.2}{(0,0)}
      \drawPath{0,1,0,1,1,1,0,1,1,0,0,1,1,1,0,0}{3}{black}{(0,0)}
      \drawPath{0,2,1,2,1,0,2,2,0,2,2,2}{2}{blue, dotted}{(0,0)}
      \drawPath{0,2,1,2,1,2,2,2,2,2,0,0}{2}{red}{(0,0)}
      \node at ( 1.5, 1.4) {$3$};
      \node at ( 2.5, 4.4) {$7$};
      \node at ( 3.5, 6.4) {$10$};
      \node at ( 4.5, 6.4) {$11$};
      \node at ( 5.5, 9.4) {$15$};
      \node at ( 6.5, 9.4) {$16$};
      \draw[line width=6, opacity=0.5] (0,0) -- (1,0);
      \draw[fill=black] (0, 0) circle (.15);
      \draw[fill=black] (7, 9) circle (.15);
      \draw[fill=black, opacity=0.5] (10,6) circle (.15);
    \end{tikzpicture}
  }
  \caption{\label{fig:modificationlemma}Two possible replacements of north steps starting on the marking path.}
\end{figure}
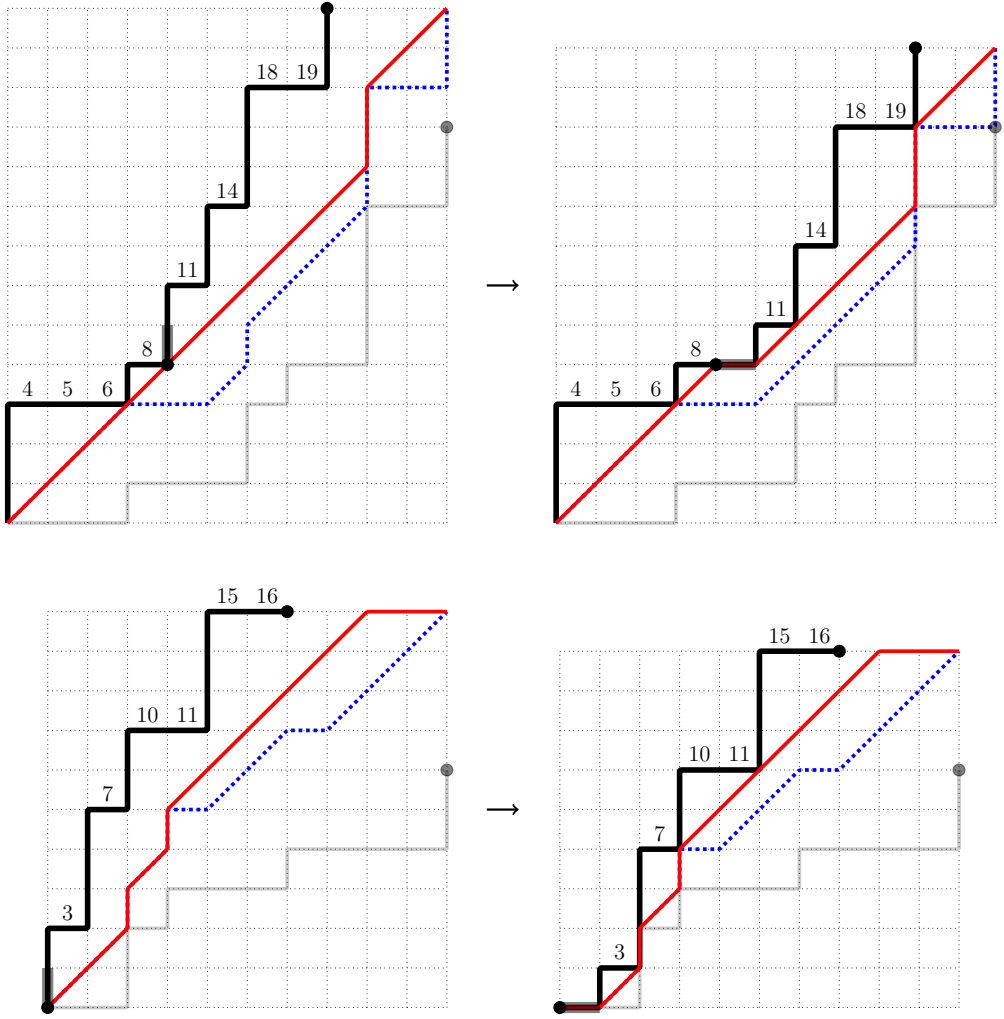

\begin{proposition}
\label{prop:constructpathfromstathard}
  Fix $n,d < d'$ and $J \subseteq \{1,\dots,n\}$.
  Let $L \in \latpaths{n,d'}$ such that the unique path $E \in \latpaths{n,\ell}$ with $\east{C} = J$ is weakly above~$L$.
  Set~$k$ to be the smallest nonnegative integer such that the point $(\ell+k,n-\ell)$ is contained in $\mar{L}{E}$.
  Then there exists~$C \in \latpaths{n,d}$ with $\stat_L(C) = J$ if and only if the two conditions
  \[
  \stat_L(E) = J \quad\text{and}\quad\ell \leq d \leq \ell+k.
  \]
  are both satisfied.
  Moreover, the path with this property is unique in this case.
\end{proposition}
  
\begin{proof}
  We have seen in \Cref{lem:suffcondhard} that these two conditions are both necessary and sufficient for such a path to exist.
  The uniqueness is also a direct consequence of the previous lemma because the path~$E$ is obviously the unique path in $\latpaths{n,\ell}$ with $\stat_L(E) = J = \east{E}$.
\end{proof}
  
One may easily check that \Cref{prop:constructpathfromstathard} indeed reduces to \Cref{prop:constructpathfromstatsimple} for the trivial lower bound.
The parameter~$k$ introduced in the previous proposition is the horizontal distance between the path~$C$ and its marking path~$\mar{L}{C}$ at its final height $n-\ell$.
In the example on the bottom left in \Cref{fig:modificationlemma} this parameter is $k = 2$.

For the following treatment, we set
\begin{align*}
  \paths{U,L}^{\e} &= \paths{U,L}^{\e,-} \sqcup \paths{U,L}^{\e,+} \\
  \paths{U,L}^{\e}_\circ &= \paths{U,L}^{\e,-}_\circ \sqcup \paths{U,L}^{\e,+}_\circ
\end{align*}
where
\[
  \paths{U,L}^{\e,-} = \{ C \in \paths{U,L}^{\e} \mid n \in \stat_L(C) \}
\]
and $\paths{U,L}^{\e,-}_\circ$ are the paths in $\paths{U,L}^{\e,-}$ obtained by removing the final east step, and $\paths{U,L}^{\e,+}$ and $\paths{U,L}^{\e,+}_\circ$ are their complements.
Observe here that we do not provide such a simple criterion to describe $\paths{U,L}^{\e,-}$ as we did in~\eqref{eq:Ltrive}.

\begin{corollary}
\label{prop:intersectionhard}
  We have
  \[
  \bigset{\stat_L(C)}{C \in \paths{U,L}^{\e,-}_\circ} = \bigset{\stat_L(C)}{C \in \paths{U,L}^{\e}_\circ} \cap \bigset{\stat_L}{\paths{U,L}^{\n}_\circ}.
  \]
\end{corollary}

\begin{proof}
  Let $C \in \paths{U,L}$ and let $C_\circ$ be the path obtained from~$C$ by removing the last step.
  We then have to show that $n \in \stat_L(C)$ if and only if $C_\circ \in \paths{U,L}^\e_\circ$ and there exists a path~$D_\circ \in \paths{U,L}^\n_\circ$ with $\stat_L(C_\circ) = \stat_L(D_\circ)$.

  We have that $n \in \stat_L(C)$ if and only if $C_\circ \in \paths{U,L}^\e_\circ$ and $(d-1,n-d) \notin \mar{L}{C_\circ}$.
  The claim follows with \Cref{prop:constructpathfromstathard}.
\end{proof}

\begin{proof}[Proof of \Cref{thm:Lrecurrencehard}]
  Consider the decomposition
  \[
    \paths{U,L} = \paths{U,L}^{\e,-} \sqcup \paths{U,L}^{\e,+} \sqcup \paths{U,L}^{\n}.
  \]
  If $U_n = \n$, we have $\Se_\mathcal{L} = \emptyset$ and thus $S_\mathcal{L} = \Sn_\mathcal{L}$ satisfies the proposed decomposition.
  The analogous consideration holds for $L_n = \e$.
  Otherwise, we have $U_n = \e, L_n = \n$ and, by the definition of $\paths{U,L}^{\e,-}$, we obtain
  \begin{align*}
    \bigset{\stat_L(C)}{C \in \paths{U,L}} =\ 
      &
      \bigset{ \stat_L(C) \cup \{n\} }{ C \in \paths{U,L}^{\e,-}_\circ}\ \cup
      \\
      &
      \bigset{ \stat_L(C) \hspace*{32.5pt} }{ C  \in \paths{U,L}^{\e,+}_\circ }\ \cup
      \\
      &
      \bigset{ \stat_L(C) \hspace*{32.5pt} }{ C \in \paths{U,L}^{\n}_\circ }.
  \end{align*}
  \Cref{prop:intersectionhard} now implies \Cref{thm:Lrecurrencehard}.
\end{proof}

\bibliographystyle{alpha}
\bibliography{EngstroemSanyalStump}

\end{document}